\documentclass[12pt]{amsart}
\usepackage{amsmath,amssymb,amsthm,wasysym,tikz-cd,stmaryrd,hyperref}
\usepackage{fullpage}
\newtheorem{theorem}{\textbf{Theorem}}
\newtheorem{proposition}{\textbf{Proposition}}
\newtheorem{conjecture}{\textbf{Conjecture}}
\newtheorem{problem}{\textbf{Problem}}

\newtheorem{lemma}{\textbf{Lemma}}
\newcommand\supp{{\operatorname{supp}}}
\newcommand\GL{{\operatorname{GL}}}
\newcommand\Sp{{\operatorname{Sp}}}
\newcommand\vmin{{v_{\operatorname{min}}}}
\newcommand\z{{\mathbf{z}}}

\title{Matrix coefficients of intertwining operators and the Bruhat order}
\author{Daniel Bump}
\address{Department of Mathematics, Stanford University, Stanford, CA 94305-2125}
\email{bump@math.stanford.edu}
\author{B\'{e}atrice Chetard}
\email{bchetard@gmail.com}

\begin{document}
\begin{abstract}
Let $(\pi_\z,V_\z)$ be an unramified principal series representation
of a reductive group over a nonarchimedean local field, parametrized
by an element $\z$ of the maximal torus in the Langlands dual group.
If $v$ is an element of the Weyl group $W$, then the standard
intertwining integral $\mathcal{A}_v$ maps $V_\z$ to $V_{v\z}$.
Letting $\psi^\z_w$ with $w\in W$ be a suitable basis of the
Iwahori fixed vectors in $V_\z$, and $\widehat\psi^\z_w$ a basis
of the contragredient representation, we define $\sigma(u,v,w)$
(for $u,v,w\in W$) to be $\langle \mathcal{A}_v\psi_u^\z,\widehat\psi^{v\z}_w\rangle$.
This is an interesting function and we initiate its study. We
show that given $u$ and $w$, there is a minimal $v$ such that
$\sigma(u,v,w)\neq 0$. Denoting this $v$ as $\vmin=\vmin(u,w)$, we will
prove that $\sigma(u,\vmin,w)$ is a polynomial of the cardinality
$q$ of the residue field. Indeed if $v>\vmin$, then $\sigma(u,v,w)$
is a rational function of $\z$ and $q$, whose denominator we
describe. But if $v=\vmin$, the dependence on $\z$ disappears.
We will express $\sigma(u,v_\vmin,w)$
as the Poincar\'e polynomial of a Bruhat interval. The proof
leads to fairly intricate considerations of the Bruhat order.

Thus our results require us to prove some facts that may be of
independent interest, relating the Bruhat
order $\leqslant$ and the weak Bruhat order $\leqslant_R$.
For example we will prove (for finite Coxeter groups) the following ``mixed meet''
property. If $u, w$ are elements of $W$, then there exists 
a unique element $m \in W$ that is maximal
with respect to the condition that $m \leqslant_R u$ and $m \leqslant w$. Thus
if $z \leqslant_R u$ and $z \leqslant w$, then $x \leqslant m$.
The value $\vmin$ is $m^{-1}u$.
\end{abstract}
\maketitle

\section{Introduction}

The formula of Macdonald~{\cite{MacdonaldSpherical}} for the
spherical functions on a reductive $p$-adic group $G (F)$ over a
nonarchimedean local field $F$ plays an important role in the
harmonic analysis on $G (F)$. Let $(\pi, V)$ be a representation of the
unramified principal series, and let $(\hat{\pi}, \hat{V})$ be its
contragredient. The representation $\pi$ may be parametrized by
Langlands-Satake parameters $\mathbf{z}$ residing in a complex torus
$\hat{T} (\mathbb{C})$, and we will thus write $(\pi, V) =
(\pi_{\mathbf{z}}, V_{\mathbf{z}})$.

The space $V$ contains a {\textit{spherical}} vector $\phi^{\circ}$, unique up
to scalar multiple, that is invariant under the standard maximal compact
subgroup $K$; similarly let $\hat{\phi}^{\circ}$ be the spherical vector in
$\hat{V}$. The Macdonald spherical function is
\[ \varsigma (g) = \langle \pi_{\mathbf{z}} (g) \phi^{\circ},
   \hat{\phi}^{\circ} \rangle . \]
Macdonald's formula identifies the value $\varsigma (g)$ with a symmetric
function of $\mathbf{z}$ that for $G = \GL_r$ is a Hall-Littlewood polynomial.

In $1980$ Casselman~{\cite{CasselmanSpherical}} gave a new proof of the
Macdonald formula. This proof involved two ingredients:

\begin{itemize}
\item Even though the Macdonald formula concerns just $\phi^{\circ}$ in the
one-dimensional space $V^K$ of $K$-invariants, it is useful to work with
the larger space $V^J$ of vectors invariant under the {\textit{Iwahori
subgroup}} $J$ of $K$;
\item If $w$ is an element of the Weyl group $W$ there is an {\textit{intertwining
integral operator}} $\mathcal{A}_w : V_{\mathbf{z}} \longrightarrow
V_{w\mathbf{z}}$; systematic use of these and related linear functionals was
another key ingredient of the proof.
\end{itemize}

These two ingredients have played an important role in subsequent theory.
Thus it is very natural to investigate
\[ \sigma (u, v, w) = \langle \mathcal{A}_v \psi^{\mathbf{z}}_u,
   \hat{\psi}^{v\mathbf{z}}_w \rangle, \]
where $u,v,w$ are elements of $W$. Here
$\psi_u^{\mathbf{z}}$ runs through a basis of Iwahori fixed vectors
$V_{\mathbf{z}}^J$ and $\hat{\psi}^{v\mathbf{z}}_w$ runs through a basis
of the Iwahori fixed vectors $\hat{V}_{v\mathbf{z}}^J$ in the contragredient
representation $(\hat{\pi}, \hat{V}_{v\mathbf{z}})$.

To describe these bases, let $\phi_w^{\mathbf{z}}$ be the element of
$V^J_{\mathbf{z}}$ whose restriction to $K$ is the characteristic function
of the double coset $J w J$. We may identify $\hat{V}_{\mathbf{z}}$ with
$V_{\mathbf{z}^{- 1}}$. (See Section~\ref{sec:padic} for further
information.) Then we have a corresponding basis $\phi_w^{\mathbf{z}^{- 1}}$
of $\hat{V}_{\mathbf{z}}$. Then we define
\[ \psi_w^{\mathbf{z}} = \bigcup_{u \geqslant w} \phi_u^{\mathbf{z}},
   \qquad \hat{\psi}_w^{\mathbf{z}} = \bigcup_{y \leqslant w}
   \phi_y^{\mathbf{z}^{- 1}} . \]

It may seem inconsistent that in $\psi_w^{\mathbf{z}}$ we sum over $u \geqslant
w$ in the Bruhat order while in $\hat{\psi}_w^{\mathbf{z}}$ we sum over $y
\leqslant w$. The explanation for this is that it leads to $\sigma (u, v, w)$
with the nicest combinatorial properties. For example let us consider the
simplest case, where $v = 1$. Then
\[ \langle \phi^{\mathbf{z}}_x, \hat{\phi}^{\mathbf{z}}_y \rangle =
   \delta_{x y} q^{\ell (x)} \]
where $q$ is the cardinality of the residue field and $\ell$ is the length
function on the Weyl group. Therefore
\[ \sigma (u, 1, w) = \sum_{\substack{
     x \geqslant u\\
     y \leqslant w}}
   \langle \phi^{\mathbf{z}}_x, \hat{\phi}^{\mathbf{z}}_y
   \rangle = \sum_{u \leqslant x \leqslant w} q^{\ell (x)} . \]
This vanishes unless $u \leqslant w$, so assume this. We recognize $\sigma (u,
1, w)$ as the Poincar{\'e} polynomial of the Bruhat interval
$[u,w]=\{x\in W|u\leqslant x\leqslant w\}$.

A second case where $\sigma$ has already appeared in practice is if $w = 1$.
We will discuss this next.

Casselman made use of the linear functionals $\phi \mapsto \mathcal{A}_w \phi
(1)$ on $V^J_{\mathbf{z}}$, and the basis $f_w$ of $V^J_{\mathbf{z}}$ is
called the {\textit{Casselman basis}}. These were used in
Casselman~{\cite{CasselmanSpherical}} and Casselman and
Shalika~{\cite{CasselmanShalika}} in order to prove the Macdonald formula for
the spherical function $\varsigma$ already mentioned, and the
Casselman-Shalika formula; this technique has been applied in many subsequent
papers.

As Casselman pointed out, computing the Casselman basis explicitly is a
difficult problem; fortunately to prove the Macdonald and Casselman-Shalika
formulas, only $f_{w_0}$ needs to be known explicitly, since the $f_w$ appear
in an integrated form, and once one integral is known, the others may be
inferred through functional equations.

Still, the $f_w$ are interesting objects. To study them, Bump and
Nakasuji~{\cite{BumpNakasuji,BumpNakasujiKL}} examined the values of the
functionals $\phi \mapsto \mathcal{A}_w \phi (1)$ introduced by Casselman on
the basis $\psi_w^{\mathbf{z}}$ and defined
\begin{equation}
\label{muvdef}
m_{u, v} =\mathcal{A}_v \psi_u (1) .
\end{equation}
About these they found interesting results and conjectures, which were refined
by Nakasuji and Naruse~{\cite{NakasujiNaruse,NaruseHook}}; the conjectures were
eventually proved using deep methods of algebraic geometry by Aluffi,
Mihalcea, Sch{\"u}rmann and Su~{\cite{AMSS}}. See also~\cite{LeeLenartEtAl}.

To connect this with the present
study, note that the functional $\mathcal{A}_w \phi (1)$ in (\ref{muvdef})
is the vector $\hat{\psi}_1^{\mathbf{z}} = \phi_1^{\mathbf{z}^{- 1}}$ in the
contragredient representation. Therefore in the notation of the present paper
\[ m_{u, v} = \sigma (u, v, 1) . \]

Now $m_{u, v}$ vanishes unless $u \leqslant v$ in the Bruhat order, and
assuming this, Bump and Nakasuji~{\cite{BumpNakasujiKL}} proved that
\[ m_{u, v} = \sum_{u \leqslant x \leqslant v} \overline{r_{x, v}}
   (\mathbf{z}) \]
where $r_{x, v} (\mathbf{z})$ are certain deformations of the
Kazhdan-Lusztig R-polynomials, and the ``bar'' means that $q$ is replaced by
$q^{- 1}$ in this formula. A similar formula is also stated in Nakasuji and
Naruse~{\cite{NakasujiNaruse,NaruseHook}}.

It was conjectured in
{\cite{BumpNakasuji,BumpNakasujiKL}} and proved in~{\cite{AMSS}} that the
denominator of $m_{u, v}$ has the form
\[ \prod_{\alpha \in S (u, v)} (1 -\mathbf{z}^{\alpha}) \]
where $S (u, v)$ is the set of positive roots $\alpha$ of $G$ such that $u
\leqslant v \cdot r_{\alpha} < v$, where $r_{\alpha} \in W$ is the reflection
associated to $\alpha$. It follows from the Deodhar inequality
({\cite{DeodharConjecture,Carrell,Polo,Dyer}}) that
\begin{equation}
\label{deodharineq}| S (u, v) | \geqslant \ell
(v) - \ell (u).
\end{equation}
Moreover $| S (u, v) | = \ell (v) - \ell (u)$ provided the
inverse Kazhdan Lusztig polynomial $Q_{u, v} = 1$; this is the polynomial
$P_{w_0 v, w_0 u}$ in the notation of {\cite{KazhdanLusztig}}. 

Bump and Nakasuji conjectured that if $G$ is simply-laced and
the Kazhdan-Lusztig polynomial $Q_{u, v} = 1$ then
\begin{equation}
\label{gkmuv}
m_{u, v} = \prod_{\alpha \in S (u, v)} \frac{1 - q^{- 1}
   \mathbf{z}^{\alpha}}{1 -\mathbf{z}^{\alpha}} .
\end{equation}
This is a generalization of the Gindikin-Karpelevich
formula~{\cite{LanglandsEuler}}, which is the case where $u = 1$; but it is much
harder to prove than the Gindikin-Karpelevich (GK) formula. This conjecture
was extended by Nakasuji and Naruse {\cite{NaruseHook,NakasujiNaruse}} to the non-simply-laced case by
replacing the Kazhdan-Lusztig criterion with the a nonsingularity condition
for Schubert varieties. In this form, the conjecture was proved by
Aluffi, Mihalcea, Sch\"urmann and Su~\cite{AMSS} using techniques
from algebraic geometry, particularly \textit{motivic Chern classes}.

We remark also that if $w = w_0$ is the long Weyl group element, then $\sigma
(u, v, w_0)$ is connected with the theory of matrix coefficients of the form
$\langle \pi (g) \psi_u^{\mathbf{z}}, \hat{\psi}_{w_0}^{\mathbf{z}}
\rangle$, where here $\hat{\psi}_{w_0}$ is the spherical vector. 
This connects with the theory of nonsymmetric Macdonald polynomials
by work of Ion~{\cite{IonMatrix}}.

We found that investigating any question about $\sigma(u,v,w)$ leads
into subtle considerations concerning the Bruhat order $\leqslant$ and the weak
Bruhat order $\leqslant_R$. For example consider the question of whether
$\sigma (u, v, w)$ is nonzero. We are able to give a satisfactory
answer to this (Theorem~\ref{maintheorem} below), 
but to explain it we must explain a relationship
between the two orders.

Recall (\cite[Chapter 3]{BjornerBrenti}) that $u \leqslant_R w$ if $\ell (w) =
\ell (u) + \ell (u^{- 1} w)$, which is the same as saying that $w$ has a
reduced expression $w = s_{i_1} \cdots s_{i_r}$ such that an initial segment
$s_{i_1} \cdots s_{i_k}$ represents $u$. The left order $\leqslant_L$ is
defined similarly.

If $u, w \in W$ then there exists a ``meet'' $x \in W$ for the weak order,
making the Weyl group into a meet semilattice. This means that \ $x
\leqslant_R u$ and $x \leqslant_R w$, and $x$ is the unique maximal element
with this property, so if $z \leqslant_R u$ and $z \leqslant_R w$ then $z
\leqslant_R x$.

The strong Bruhat order $\leqslant$ does not have this property. For example
in Cartan type~$A_2$ if $u = s_1 s_2$ and $w = s_2 s_1$ then
\[ \{ z \in W|z \leqslant u, w \} = \{ 1, s_1, s_2 \} \]
and this set does not have a unique maximal element.

However we will show (for finite Coxeter groups $W$) that $u, w$ have a
{\textit{mixed meet}} for the Bruhat and weak orders. We will prove below in
Theorem~\ref{thm:mixedmeet} that if $u, w$ are elements of $W$, a finite
Coxeter group, then there exists a unique element $m \in W$ that is maximal
with respect to the condition that $m \leqslant_R u$ and $m \leqslant w$. Thus
if $z \leqslant_R u$ and $z \leqslant w$, then $z \leqslant m$. We will call
$m$ the {\textit{mixed meet}} of $u$ and $w$ with respect to the weak and strong
Bruhat orders.

Now let $\vmin = \vmin (u, w) = m^{- 1} u$. We will prove:

\begin{theorem}
  \label{maintheorem}
  We have $\sigma (u, v, w) = 0$ unless $v \geqslant \vmin (u, v)$. If $v = \vmin
  (u, v)$ then $\sigma (u, v, w)$ is a polynomial in $q$, independent
  of~$\mathbf{z}$. More precisely,
  \begin{equation}
  \label{sigmabruhat}
  \sigma(u,\vmin(u,v),w)=q^{-\ell(v)}\sum_{z\in[u,wv]}q^{\ell(z)}.
  \end{equation}
\end{theorem}

Let us denote $\sigma_0 (u, w) = \sigma (u, \vmin(u,w), w)$.

There is one more interesting phenomenon that we only partially
understand. For many $(u,v,w)$ we find that $\sigma(u,v,w)$ has
a special form. In these cases
\begin{equation}
\label{gkform}
\sigma(u,v,w)=\sigma_0(u,w)\,\prod_{\alpha\in S(u,v,w)}\frac{1-q^{-1}\z^{\alpha}}{1-\z^\alpha}
\end{equation}
where we recall that $\sigma_0(u,w)$ is a polynomial in $q$. Here $S(u,v,w)$
is the set of positive roots $\alpha$ defined in (\ref{suvwdef})
below. This formula generalizes (\ref{gkmuv}), which we have
already mentioned
resembles the \textit{Gindikin-Karpelevich formula} which is the
case $w=1$. Therefore if (\ref{gkform}) is valid we will say that
the triple $(u,v,w)$ is \textit{of GK type}.
In contrast with that special case,
we do not know precisely which of $(u,v,w)$ are of GK type.
But if $w=1$, so $\sigma(u,v,1)=m_{u,v}$, we have already
explained that these have been determined in
\cite{BumpNakasuji,BumpNakasujiKL,NakasujiNaruse,NaruseHook,AMSS}.

If $G=\GL(3)$, so $W=S_3$, there are 167
triples $(u,v,w)$ with $v\geqslant \vmin(u,w)$. (Recall that these are the
cases where $\sigma(u,v,w)\neq 0$.) Of these, all but 20 are
of GK type. The exceptional triples $(u,v,w)$ such that $\sigma(u,v,w)$
is \textit{not} of the form (\ref{gkform}) are:
\[
\begin{array}{rcccl}
(s_{1}s_{2},s_{1}s_{2}s_{1},s_{1}s_{2})&\qquad&
(s_{1}s_{2},s_{1}s_{2},s_{1}s_{2})&\qquad&
(s_{2}s_{1},s_{1}s_{2}s_{1},s_{2}s_{1})\\
(s_{2}s_{1},s_{2}s_{1},s_{2}s_{1})&\qquad&
(s_{1},s_{1}s_{2}s_{1},s_{1}s_{2})&\qquad&
(s_{1},s_{2}s_{1},s_{1}s_{2})\\
(s_{1},s_{1}s_{2},s_{1}s_{2})&\qquad&
(s_{1},s_{1},s_{1}s_{2})&\qquad&
(s_{1},s_{1}s_{2}s_{1},s_{2}s_{1})\\
(s_{1},s_{2}s_{1},s_{2}s_{1})&\qquad&
(s_{1},s_{1}s_{2}s_{1},s_{1})&\qquad&
(s_{1},s_{2}s_{1},s_{1})\\
(s_{2},s_{1}s_{2}s_{1},s_{1}s_{2})&\qquad&
(s_{2},s_{1}s_{2},s_{1}s_{2})&\qquad&
(s_{2},s_{1}s_{2}s_{1},s_{2}s_{1})\\
(s_{2},s_{2}s_{1},s_{2}s_{1})&\qquad&
(s_{2},s_{1}s_{2},s_{2}s_{1})&\qquad&
(s_{2},s_{2},s_{2}s_{1})\\
(s_{2},s_{1}s_{2}s_{1},s_{2})&\qquad&
(s_{2},s_{1}s_{2},s_{2})&\qquad& \\
\end{array}\]

If $G=\Sp(4)$ so $W$ is dihedral of order 8, there are 401 triples
$(u,v,w)$ such that $v\geqslant\vmin(u,w)$, and of these (\ref{gkform})
is satisfied by 305.
If $G=\GL(4)$, so $S_4$, there are 9597 such triples,
and of these, (\ref{gkform}) is satisfied by 6281. 

\begin{problem}
Determine for which $(u,v,w)$ equation (\ref{gkform}) is satisfied.
\end{problem}

\medbreak
\textit{Acknowledgements:}
We used Sage Mathematical Software~\cite{sagemath} to carry out this research.
This work was partially supported by NSF grant DMS-1601026. We are
grateful to Valentin Buciumas, Eric Marberg, and Maki Nakasuji for helpful 
comments. Especially, Eric Marberg gave us very useful advice about
technical points of the Bruhat order.

\section{The Bruhat order and the weak order\label{sec:bruhat}}

Let $W$ be a Coxeter group. After a preliminary result that is valid for
general Coxeter groups we will assume that $W$ is finite. Let $\{ s_1, \cdots,
s_r \}$ be the simple reflections in $W$. Let $w_0$ denote its longest
element. It is probable that some of our results apply more generally to
arbitrary Coxeter groups but we will make use of $w_0$ in the proofs.

We will denote by $\leqslant$ the usual (strong) Bruhat order on $W$. We will
also denote by $\leqslant_R$ the weak Bruhat order. The weak Bruhat order
$\leqslant_R$ can be defined in terms of the length function $\ell$ on $W$ by
$u \leqslant_R v$ if $\ell (u) + \ell (u^{- 1} v) = \ell (v)$.
An equivalent definition is that $u\leqslant_R v$ if there exists
a reduced word $s_{i_1}\cdots s_{i_n}$ for $v$ such that the initial
segment $s_{i_1}\cdots s_{i_k}$ is a reduced word for $u$, for some $k\leqslant n$.
See~{\cite{BjornerBrenti}} Proposition~3.1.2 for other characterizations.

We will make use of two versions of the Hecke algebra of $W$. The first has a
basis $T_w$ ($w \in W$) such that if $s = s_i$ is a simple reflection and $s w
> w$ then $T_s T_w = T_{s w}$ and
\[ T_s T_{s w} = (q - 1) T_{s w} + q T_w . \]
See {\cite{HumphreysCoxeter}} Theorem~7.1 for the construction of
$\mathcal{H}_q$. We will denote $T_{s_i} = T_i$.

We will also use a variant $\mathcal{H}_0$ which is the same as
$\mathcal{H}_q$ with $q$ specialized to $0$. We will use $U_w$ to represent
the basis elements, so
\[ U_s U_w = \left\{\begin{array}{ll}
 U_{s w} & \text{if $s w > w$,}\\
 U_w & \text{if $s w < w$} .
\end{array}\right. \]
Again $U_i = U_{s_i}$.

The basis elements $U_w$ of $\mathcal{H}_0$ form a monoid $\mathcal{M}_0$, and
$\mathcal{H}_0$ is the algebra of the monoid $\mathcal{M}_0$. We will
call $\mathcal{M}_0$ the {\textit{Demazure}} monoid. Since $i : W
\longrightarrow \mathcal{M}_0$ defined by $i (w) = U_w$ is a bijection, we may
transfer the multiplication in $\mathcal{M}_0$ back to $W$. This
multiplication $\circ$ on $W$ is sometimes called the {\textit{Demazure}}
multiplication.

Define maps $u_i^{\vee}, u_i^{\wedge} : W \longrightarrow W$ by
\[ u_i^{\wedge} (v) = \left\{\begin{array}{ll}
 s_i v & \text{if $s_i v > v$,}\\
 v & \text{if $s_i v < v$} .
\end{array}\right. \qquad u_i^{\vee} (v) = \left\{\begin{array}{ll}
 v & \text{if $s_i v > v$,}\\
 s_i v & \text{if $s_i v < v$} .
\end{array}\right. \]
\begin{lemma}
\label{monoidbraid}Suppose that $s_i s_j$ has finite order. Then
$u^{\wedge}_i$ and $u^{\vee}_i$ satisfy the braid relation $u_i u_j \cdots =
u_j u_i \cdots$, where the number of terms on both sides is the order of
$s_i s_j$. 
\end{lemma}

\begin{proof}
Let $D_{i j}$ be the dihedral group generated by $s_i$ and $s_j$. Applied to
any $v \in W$, it is easy to see that \ $u_i^{\vee} u^{\vee}_j \cdots$ and
$u^{\vee}_j u^{\vee}_i \cdots$ both produce the shortest element of the
coset $D_{i j} v$ in at most $n_{i j}$ steps, where $n_{i j}$ is the order
of $s_i s_j$. Hence $u^{\vee}_i u^{\vee}_j \cdots (v) = u^{\vee}_j
u^{\vee}_i \cdots (v)$. The braid relation for $u_i^{\wedge}$ is proved
similarly.
\end{proof}

If $u, v \in W$ define
\[ U_i \uparrow v = \left\{\begin{array}{ll}
 s_i v & \text{if $s_i v > v$,}\\
 v & \text{if $s_i v < v$,}
\end{array}\right. \qquad U_i \downarrow v = \left\{\begin{array}{ll}
 v & \text{if $s_i v > v$,}\\
 s_i v & \text{if $s_i v < v$,}
\end{array}\right. \]
\[ v \uparrow U_i = \left\{\begin{array}{ll}
 v s_i & \text{if $v s_i > v$},\\
 v & \text{if $v s_i < v$},
\end{array}\right. \qquad v \downarrow U_i = \left\{\begin{array}{ll}
 v & \text{if $v s_i > v$},\\
 v s_i & \text{if $v s_i < v$} .
\end{array}\right. \]
\begin{proposition}
The operations $U_i : v \longrightarrow U_i \uparrow v$, $U_i : v
\longrightarrow U_i \downarrow v$ extend to left actions of the monoid
$\mathcal{M}_0$ on $W$, meaning that
\begin{equation}
\label{upaction} (U U') \uparrow v = U \uparrow (U' \uparrow v), \qquad (U
U') \downarrow v = U \downarrow (U' \downarrow v), \qquad U, U' \in
\mathcal{M}_0, v \in W.
\end{equation}
Similarly \ $U_i : v \longrightarrow v \uparrow U_i$, $U_i : v
\longrightarrow v \downarrow U_i$ extend to right actions of $\mathcal{M}_0$
such that
\begin{equation}
\label{downaction} v \uparrow (U U') = (v \uparrow U) \uparrow U', \qquad
v \downarrow (U U') = (v \downarrow U) \downarrow U', \qquad U, U' \in
\mathcal{M}_0, v \in W.
\end{equation}
\end{proposition}

\begin{proof}
The algebra $\mathcal{H}_0$ may be characterized as the algebra generated by
the $U_i$ subject to the quadratic relations $U_i^2 = U_i$ and the braid
relations. Taking $u_i = u_i^{\vee}$ (resp.~$u_i = u_i^{\wedge}$) these
relations are satisfied by the $u_i$ by Lemma~\ref{monoidbraid}, and so we
have a representation of $\mathcal{M}_0$ on $W$ satisfying (\ref{upaction}).
The proof of (\ref{downaction}) is similar.
\end{proof}

Since $w \mapsto U_w$ is a bijection $i : W \longrightarrow \mathcal{M}_0$,
the operations $\uparrow$ correspond to the left and right regular operations
of $\mathcal{M}_0$ on itself. That is, we can define $U_u \uparrow v = i^{- 1}
(U_u U_v) = u \uparrow U_v$ and it is clear that this operation has the
advertised properties. Moreover $i (U_u \uparrow v) = U_u U_v = i (u \uparrow
U_v)$, and therefore in terms of the Demazure multiplication,
\begin{equation}
\label{demsymmetry} U_u \uparrow v = u \uparrow U_v = u \circ v\,.
\end{equation}

For the rest of this section we specialize to the case that $W$ is a finite
Coxeter group. In this case $W$ has a longest element $w_0$, then $u \mapsto u
w_0$ is an order-reversing involution of $W$. From this it follows that
\begin{equation}
\label{worelation} 
w_0 (u \uparrow U_v) = (w_0 u) \downarrow U_v , \qquad 
(U_u \uparrow v) w_0 = U_u \downarrow (v w_0)
\end{equation}
Combining (\ref{demsymmetry}) and (\ref{worelation}) gives us formulas for
$\downarrow$ in terms of the Demazure multiplication:
\begin{equation}
\label{downdemazure} u \downarrow U_v = w_0 ((w_0 u) \circ v), \qquad U_u
\downarrow v = (u \circ (v w_0)) w_0
\end{equation}

Let $(i_1, \cdots, i_m)$ be a sequence of indices. We will say that the
sequence {\textit{contains}} $w \in W$ if $s_{j_1} \cdots s_{j_k}$ is a reduced
expression for $w$ for some subsequence $(j_1, \cdots, j_k)$ of $(i_1, \cdots,
i_m)$. Knutson and Miller~{\cite{KnutsonMillerSubword}} characterized this condition
in terms of the Demazure product $\circ$ on $W$ as follows:

\begin{proposition}
\label{KMLemma}The sequence $(i_1, \cdots, i_m)$ contains $w$ if and only if
$w \leqslant s_{i_1} \circ \cdots \circ s_{i_m}$.
\end{proposition}

\begin{proof}
See~{\cite{KnutsonMillerSubword}} Lemma~3.4.
\end{proof}

As a consequence we deduce a monotonicity property of $\circ$.

\begin{proposition}
\label{monotonedemazure}If $u, u', v, v' \in W$ and $u \leqslant u'$, $v
\leqslant v'$ then $u \circ v \leqslant u' \circ v'$.
\end{proposition}

\begin{proof}
Let $(i_1', \cdots, i_{m'}')$ and $(j_1', \cdots, j_{n'}')$ be reduced words
for $u'$ and $v'$. Then we may find subwords $(i_1, \cdots, i_m)$ and $(j_1,
\cdots, j_n)$ that are reduced words for $u$ and $v$. By Proposition~2, $u'
\circ v'$ is the maximal element of $W$ that is contained in $(i_1', \cdots,
i_{m'}', j_1', \cdots, j_{n'}')$ and $u \circ v$ is the maximal element
contained in $(i_1, \cdots, i_m, j_1, \cdots, j_n)$. Since the second
sequence is a subsequence of the first, we obtain the stated monotonicity
property.
\end{proof}

From Proposition~\ref{monotonedemazure} we may infer two monotonicity
properties of the ``down'' actions of~$\mathcal{H}_0$.

\begin{proposition}
\label{downmonotone}If $x \geqslant u$ then $U_w \downarrow x \geqslant U_w
\downarrow u$ and $x \downarrow U_w \geqslant u \downarrow U_w$.
\end{proposition}

\begin{proof}
Using (\ref{downdemazure}) the inequality that we need for $U_w \downarrow x
\geqslant U_w \downarrow u$ is
\begin{equation}
\label{tmpineq} (w \circ (x w_0)) w_0 \geqslant (w \circ (u w_0)) w_0 .
\end{equation}
We have $x w_0 \leqslant u w_0$ and so by Proposition~\ref{monotonedemazure}
we have $w \circ (x w_0) \leqslant w \circ (u w_0)$. The inequality
(\ref{tmpineq}) follows. The second inequality is proved the same way.
\end{proof}

\begin{proposition}
\label{downmonotonebis}If $w \geqslant v$ then $U_w \downarrow x \leqslant
U_v \downarrow x$ and $x \downarrow U_w \leqslant x \downarrow U_v$ for any $w \in W$.
\end{proposition}

\begin{proof}
By Proposition~\ref{monotonedemazure} we have $w \circ x w_0 \geqslant v
\circ x w_0$ and so $(w \circ x w_0) w_0 \leqslant (v \circ xw_0) w_0$. The
first inequality now follows from (\ref{downdemazure}), and the other is
proved the same way.
\end{proof}

If $u \leqslant v$ in the Bruhat order let $[u, v]$ denote the Bruhat interval
$\{ x \in W|u \leqslant x \leqslant v \}$.

\begin{theorem}
  \label{translatemin}The translated Bruhat interval $u [1, v]$ has maximal
  and minimal elements $u \circ v$ and $u \downarrow U_v$. Similarly $[1, u]
  v$ has maximal and minimal elements $u \circ v$ and $U_u \downarrow v$.
  Moreover $[u, w_0] v$ has minimal element $u \downarrow U_v$ and $u [v,
  w_0]$ has minimal element $U_u \downarrow v$.
\end{theorem}

\begin{proof}
  Let us start by showing that $u\circ v\in u[1,v]$.
  This may be proved by induction on $\ell (v)$: we write $v =
  v' s_i$ where $\ell (v') = \ell (v) - 1$. By induction $u \circ v' \in u [1,
  v']$. Now $u \circ v$ is either $u \circ v'$ or $(u \circ v') s$ and in
  either case $u \circ v \in u [1, v]$.

  We have shown that $u\circ v\in u[1,v]$, but we need to show that it
  is the maximal element. 
  Let us pick reduced words $(i_1, \cdots, i_k)$ and $(j_1, \cdots, j_l)$ for
  $u$ and $v$, respectively. If $w \in u [1, v]$ then we may write $w =
  s_{i_1} \cdots s_{i_k} s_{j_1'} \cdots s_{j_m'}$ where $(j_1', \cdots,
  j_m')$ is a subsequence of $(j_1, \cdots, j_l)$ and so the sequence $(i_1,
  \cdots, i_k, j_1, \cdots, j_l)$ contains $w$. By Proposition~\ref{KMLemma}
  we have
  \[ w \leqslant s_{i_1} \circ \cdots \circ s_{i_k} \circ s_{j_1} \circ \cdots
     \circ s_{j_l} = u \circ v, \]
  where we have used the fact that $(i_1, \cdots, i_k)$ and $(j_1, \cdots,
  j_l)$ are reduced words, so $u = s_{i_1} \circ \cdots \circ s_{i_k}$ and $v
  = s_{j_1} \circ \cdots \circ s_{j_l}$. 
  This proves that $u [1, v]$ has the unique maximal element $u \circ v$.
  
  Now we make use of the fact that $u \mapsto w_0 u$ is an order reversing
  bijection of $W$ to deduce that the Bruhat interval $w_0 u [1, v]$ has a
  minimal element $w_0  (u \circ v)$. Replacing $u$ by $w_0 u$ and using
  (\ref{downdemazure}) we see that $u [1, v]$ has a minimal element $w_0 
  ((w_0 u) \circ v) = u \downarrow U_v$.
  
  Now let us look at $u [v, w_0]$. Since right multiplication by $w_0$ is
  order reversing and maps $[v, w_0]$ to $[1, v w_0]$ we have
  \[ \min (u [v, w_0]) = (\max (u [1, v w_0])) w_0 = (u \circ (v w_0)) w_0 = 
  U_u \downarrow v . \]

  The remaining statements about $[1, u] v$ and $[u, w_0] v$ are the mirror
  images of cases already proved and may be established the same way.
\end{proof}

\begin{lemma}
\label{lemsix}Suppose that $u' \leqslant_R u$ and $w \in W$. Then $u'
\downarrow U_w \leqslant_R u$.
\end{lemma}

\begin{proof}
Suppose we know this in the case where $w = s$ is a simple reflection. Then
writing $w = s_{i_1} \cdots s_{i_k}$ we have $u' \downarrow U_w = u'
\downarrow U_{s_i} \cdots U_{s_k}$ and so applying the case of a simple
reflection repeatedly we get $u' \downarrow U_w \leqslant u$.

Thus we are reduced to the case where $w = s$ is a simple reflection. If $u'
s > u'$ then $u' \downarrow U_s = u'$ in which case we are done. Therefore
we may assume that $u' s < s$. Then we may choose a reduced word for $u'$ of
the form $u' = s_{i_1} \cdots s_{i_r}$ where $s_{i_r} = s$. Let $s_{j_1}
\cdots s_{j_k}$ be a reduced word for $(u')^{- 1} u$. Then since $u'
\leqslant_R u$ the expression $s_{i_1} \cdots s_{i_r} s_{j_1} \cdots s_{j_k}$
is a reduced expression for $u$. So
$u' \downarrow U_s = u' s = s_{i_1} \cdots s_{i_{r - 1}} \leqslant_R u$.
\end{proof}

\begin{lemma}
\label{lemseven}Let $u, v \in W$. Then $u \downarrow U_v \leqslant_R u$.
\end{lemma}

\begin{proof}
This is the special case $u' = u$ of Lemma~\ref{lemsix}.
\end{proof}

\begin{lemma}
\label{lemeight}If $x \leqslant_R u^{- 1}$ then $u x \leqslant_R u$.
\end{lemma}

\begin{proof}
We have $\ell (u^{- 1}) = \ell (x) + \ell (x^{- 1} u^{- 1})$. We want to
deduce $\ell (u) = \ell (u x) + \ell ((u x)^{- 1} u)$. But $\ell (u) = \ell
(u^{- 1})$, $\ell ((u x)^{- 1} u) = \ell (x)$ and $\ell (x^{- 1} u^{- 1}) =
\ell(u x)$, so this is clear.
\end{proof}

\begin{lemma}
\label{lemnine}If $u, w \in W$ then $u (u^{- 1} \downarrow U_w) \leqslant
w$.
\end{lemma}

\begin{proof}
Let $w = s_{i_1} \cdots s_{i_k}$ be a reduced expression. It is clear from
the definition of $u^{- 1} \downarrow U_w = u^{- 1} \downarrow U_{s_{i_1}}
\downarrow \cdots \downarrow U_{s_{i_k}}$ that $u^{- 1} \downarrow U_w =
u^{- 1} w'$ where $w' = s_{j_1} \cdots s_{j_k}$ and $(j_1, \cdots, j_k)$ is
a subsequence of $(i_1, \cdots, i_k)$. Then $u (u^{- 1} \downarrow U_w) = w'
\leqslant w$.
\end{proof}

\begin{proposition}
\label{weaklifting}Suppose that $w \leqslant_R w z$ and $w \leqslant_R w v$.
Moreover assume that $w z \leqslant w v$. Then $z \leqslant v$.
\end{proposition}

\begin{proof}
We argue by induction on $\ell (w)$. If $w = 1$ this is certainly true.
Otherwise, let $s$ be a left descent of $w$. Since $w \leqslant w z$ it is
also a left descent of $w z$ and similarly of $w v$. Thus we may write $w =
s w'$ and we have $w' z < s w' z$ and similarly $w' v < s w' v$. Moreover we
are assuming that $s w' z \leqslant s w' v$. Using the lifting property of
the Bruhat order (Proposition~2.2.7 in {\cite{BjornerBrenti}}), it follows
that $w' z \leqslant w' v$. By induction we have $z \leqslant w$.
\end{proof}

It is well-known that left multiplication by $w_0$ is order reversing for the
Bruhat order. The following result is a generalization, showing that under
certain circumstances the Bruhat order is reversed by an arbitrary element
$u$.

\begin{proposition}
\label{conditionalreversal}Suppose that $x \leqslant_R u^{- 1}$ and $y
\leqslant_R u^{- 1}$. Moreover assume that $x \leqslant y$. Then $u x
\geqslant u y.$
\end{proposition}

\begin{proof}
We recall that the map $a \rightarrow a w_0$ is order reversing for both the
Bruhat order and for the weak order ({\cite{BjornerBrenti}}
Propositions~2.3.4 and~3.1.5). Thus if we make the variable changes $x
\rightarrow x w_0$, $y \rightarrow y w_0$ and $u \rightarrow w_0 u$ the
statement is seen to be equivalent to:
\[ \text{If $u^{- 1} \leqslant_R x$ and $u^{- 1} \leqslant_R y$ and $y
 \leqslant x$\quad then\quad} u x \geqslant u y. \]
This follows from Proposition~\ref{weaklifting} with $u = w^{- 1}$, $x = w
v$, $y = w z$.
\end{proof}

The weak order $\leqslant_R$ has the following ``meet semilattice'' property,
which is Theorem~3.2.1 of~{\cite{BjornerBrenti}}: If $u, w \in W$ then the set
of $v$ such that $v \leqslant_R u$ and $v \leqslant_R w$ has a maximal
element. This maximal element is called the {\textit{meet}} of $u$ and $v$ for
the partial order $\leqslant_R$.

This property is not shared by the usual (strong) Bruhat order. For example in
the $A_2$ Weyl group $s_1 s_2$ and $s_2 s_1$ do not have a meet for the Bruhat
order. However we will prove a ``mixed'' meet property that combines the weak
and strong Bruhat orders.

\begin{theorem}[Mixed meet property]
\label{thm:mixedmeet}
Let $u, w$ be elements of $W$. Then the set 
\[\{ x \in W|x \leqslant_R u, x \leqslant w \}\]
has a unique maximal element $m$ for the Bruhat order. This
means that $m \leqslant_R u$, $m \leqslant w$ and if $v \leqslant_R u$, $v
\leqslant w$ then $v \leqslant m$. In terms of the $\downarrow$ operation we
have:
\[ m = u (u^{- 1} \downarrow U_w) . \]
\end{theorem}

\begin{proof}
It follows from Lemmas~\ref{lemseven} and~\ref{lemeight} that $u (u^{- 1}
\downarrow U_w) \leqslant_R u$. Moreover by Lemma~\ref{lemnine} we have $u
(u^{- 1} \downarrow U_w) \leqslant w$.

Now suppose that $v \leqslant_R u$, $v \leqslant w$. We must prove that $v
\leqslant m$. By Proposition~\ref{downmonotonebis} we have
\[ u^{- 1} \downarrow U_w \leqslant u^{- 1} \downarrow U_v . \]
Now we note that $u^{- 1} \downarrow U_w, u^{- 1} \downarrow U_v$ are both
$\leqslant_R u^{- 1}$ and so we may apply
Proposition~\ref{conditionalreversal} to see that
\[ m = u (u^{- 1} \downarrow U_w) \geqslant u (u^{- 1} \downarrow U_v) . \]
However since $v \leqslant_R u$ we have $u^{- 1} \downarrow U_v = u^{- 1} v$
proving that $m \geqslant v$, as required.
\end{proof}

\section{Matrix coefficients of intertwining integrals\label{sec:padic}}

Let $G$ be a split reductive group over a nonarchimedean local field $F$. Let
$\mathfrak{o}$ and $\mathfrak{p}$ be the ring of integers of $F$, and its
maximal ideal, and let $q = | \mathfrak{o}/\mathfrak{p} |$ be the residue
cardinality. Let $T$ be a split maximal torus and $B = T N$ a Borel subgroup
of $G$ containing $T$, with $N / U$ its unipotent radical. Let $\Phi$ be the
root system of $G$ with respect to $T$, with $\Phi^+$ the positive roots. Let
$\hat{T}$ be the dual torus in the connected Langlands dual group $\hat{G}$.
Let $\Lambda = X^{\ast} (\hat{T})$ be the weight lattice of $\hat{G}$. If
$\mathbf{z} \in \hat{T} (\mathbb{C})$ and $\lambda \in \Lambda$ we will
denote by $\mathbf{z}^{\lambda}$ the application of $\lambda$ to
$\mathbf{z}$.

The tori $T$ and $\hat{T}$ are in duality, so $\Lambda$ may be identified
with the cocharacter group $X_{\ast} (T) \cong T (F) / T (\mathfrak{o})$. Let
$\varpi$ be a generator of $\mathfrak{p}$. With $\lambda \in \Lambda$, the
image of $\varpi$ under the cocharacter $F^{\times} \longrightarrow T (F)$
corresponding to $\lambda$ will be denoted $\varpi^{\lambda}$. Then the coset
$\varpi^{\lambda} T (\mathfrak{o})$ is the image of $\lambda$ under the
isomorphism $X_{\ast} (T) \cong T (F) / T (\mathfrak{o})$.

Also $\mathbf{z}$ determines an unramified quasicharacter
$\chi_{\mathbf{z}}$ of $T (F)$ such that $\chi_{\mathbf{z}}
(\varpi^{\lambda}) =\mathbf{z}^{\lambda}$. The principal series
representation $(\pi_{\mathbf{z}}, V_{\mathbf{z}})$ is the representation
of $G (F)$ obtained from $\chi_{\mathbf{z}}$ by parabolic induction. Thus
the space $V_{\mathbf{z}}$ consists of locally constant functions that
satisfy
\[ \phi (b g) = (\delta^{1 / 2} \chi_{\mathbf{z}}) (b) \phi (g), \qquad b
   \in B (F) . \]
Here the function $\chi_{\mathbf{z}}$ is extended to $B (F)$ by means of the
homomorphism $B (F) \longrightarrow T (F)$ with kernel $U (F)$, and $\delta$
is the modular quasicharacter on $B (F)$.

Let $K = G (\mathfrak{o})$ be the standard maximal compact subgroup. Let $J$
be the Iwahori subgroup that is the preimage of $B (\mathbb{F}_q)$ under the
reduction mod $\mathfrak{p}$ homomorphism $K \rightarrow G (\mathbb{F}_q)$. We
will normalize the Haar measures on $G$ and $K$ so that $J$ has volume~$1$.

The contragredient $\hat{\pi}_{\mathbf{z}}$ of $\pi_{\mathbf{z}}$ is
$\pi_{\mathbf{z}^{- 1}}$. Indeed an invariant nondegenerate bilinear pairing
$V_{\mathbf{z}} \times V_{\mathbf{z}^{- 1}} \longrightarrow \mathbb{C}$ is
given by
\[ \label{innerp} \langle f_1, f_2 \rangle = \int_K f_1 (k) f_2 (k) \, d k. \]
Thus we will denote $\hat{V}_{\mathbf{z}} = V_{\mathbf{z}^{- 1}}$.

A basis of $J$-fixed vectors in $V_{\mathbf{z}}$ consists of the functions
$\{ \phi_w^{\mathbf{z}} \}$ with support in a single double coset $B w J$
with $w$ in the Weyl group $W = N (T) / T$. We choose the representative of a
Weyl group element $w \in W$ to be in $N (T) \cap K$ and by abuse of notation
we denote the representative by the same letter $w$; this abuse of notation is
justified by the fact that in dealing with the unramified principal series,
nothing depends on the choice of representative beyond the fact that it is in
$K$. In particular we define
\[ \phi_w^{\mathbf{z}} (b w' k) = \left\{\begin{array}{ll}
     \text{$(\delta^{1 / 2} \chi_{\mathbf{z}}) (b)$} & \text{if $w = w'$ in
     $W$,}\\
     0 & \text{otherwise},
   \end{array}\right. \]
and this definition does not depend on the choice of representative since
$\chi_{\mathbf{z}}$ is unramified. In view of our identification
$\hat{V}_{\mathbf{z}} = V_{\mathbf{z}^{- 1}}$ we will denote
$\hat{\phi}_w^{\mathbf{z}} = \phi_w^{\mathbf{z}^{- 1}}$. Since $\text{vol}(JwJ)=q^{\ell(w)}$
it follows from the definition (\ref{innerp}) that
\begin{equation}
\label{innerplength}
\langle \phi^{\mathbf{z}}_w, \hat{\phi}_{w'}^{\mathbf{z}} \rangle =
   \text{vol} (J w J) \delta_{w, w'} = q^{\ell (w)} \delta_{w, w'} . 
\end{equation}
Following {\cite{BumpNakasuji,BumpNakasujiKL,NakasujiNaruse,AMSS}} we will not
work directly with $\phi_w^{\mathbf{z}}$ but with the basis $\{
\hat{\psi}_w^{\mathbf{z}} \}$ defined by
\[ \psi_w^{\mathbf{z}} = \sum_{y \geqslant w} \phi^{\mathbf{z}}_w . \]
In the contragredient we will consider
\[ \hat{\psi}_w^{\mathbf{z}} = \sum_{y \leqslant w} \phi^{\mathbf{z}^{-
   1}}_w . \]
Note the inversion of the Bruhat order.

If $w \in W$ there is an intertwining operator $\mathcal{M}_w :
V_{\mathbf{z}} \longrightarrow V_{w\mathbf{z}}$ defined by
\[ (\mathcal{M}_w f) (g) = \int_{N \cap w N_- w^{- 1}} f (w^{- 1} n g) \, d n.
\]
The integral is convergent if $| \mathbf{z}^{\alpha} | < 1$ for positive
roots $\alpha$, and has meromorphic continuation to all $\mathbf{z}$. Our
purpose is to study the ``matrix coefficient''
\[ \sigma (u, v, w) = \langle \mathcal{A}_v \psi^{\mathbf{z}}_u,
   \hat{\psi}_w^{v\mathbf{z}} \rangle . \]

Define a functional $\Lambda_w$ on $\mathcal{H}_q$ by
\[ \Lambda_w (T_y) = \left\{\begin{array}{ll}
     q^{\ell (y)} & \text{if $y \leqslant w$,}\\
     0 & \text{otherwise.}
   \end{array}\right. \]
If $w=w_0$ then $\Lambda_w(T_y)=q^{\ell(y)}$, and it is
easy to see that in this case $\Lambda_{w_0}$ is a ring
homomorphism $\mathcal{H}_q\to\mathbb{C}$.
Define
\begin{equation}
\label{theta_def}
\Theta (x, y, w) = \Lambda_w (T_x T_{y^{- 1}})\; . 
\end{equation}

The following result was proved in {\cite{BumpNakasujiKL}}.

\begin{theorem}
  \label{rrecursion}There exist polynomials $r_{u, v} (\mathbf{z})$ of
  $\mathbf{z} \in \hat{T} (\mathbb{C})$ such that $r_{u, v} = 0$ unless $u
  \leqslant v$ and $r_{u, u} = 1$ satisfying the following recursive
  relations. Let $s$ be a simple reflection such that $s v < v$. Then
  \[ r_{u, v} (\mathbf{z}) = \left\{\begin{array}{ll}
       \frac{1 - q}{1 -\mathbf{z}^{- v^{- 1} \alpha}} r_{u, s v}
       (\mathbf{z}) + r_{s u, s v} (\mathbf{z}) & \text{if $s u < u$,}\\
       (1 - q) \frac{\mathbf{z}^{- v^{- 1} \alpha}}{1 -\mathbf{z}^{- v^{-
       1} \alpha}} r_{u, s v} (\mathbf{z}) + q r_{s u, s v} (\mathbf{z}) &
       \text{if $s u > u$.}
     \end{array}\right. \]
\end{theorem}

The Kazhdan-Lusztig R-polynomial $R_{u, v}$ is the limit of $r_{u, v}
(\mathbf{z})$ when $\mathbf{z} \longrightarrow \infty$ in such a way that
$\mathbf{z}^{\alpha} \longrightarrow \infty$ for all positive roots
$\alpha$. We note that since $s v < v$, $- v^{- 1} \alpha$ is a positive root
so $\mathbf{z}^{- v^{- 1} \alpha} \longrightarrow \infty$ under this
specialization.

The value $r_{u,v}(\mathbf{z})$ is a polynomial in $q$ and
$\mathbf{z}$. Let us denote by $\overline{r_{u,v}}(\mathbf{z})$
the result of replacing $q$ by $q^{-1}$ in this polynomial.

\begin{theorem}
  We have
\begin{equation}
  \label{sigmabnform} \sigma (u, v, w) = \sum_{\substack{
    x \geqslant u\\
    y \leqslant v}}
  q^{- \ell(y)} \Theta (x, y, w) \overline{r_{y, v}}(\mathbf{z}).
\end{equation}
\end{theorem}  

\begin{proof}
The proof follows~{\cite{BumpNakasuji,BumpNakasujiKL}}, who made use of a
technique from~{\cite{Rogawski}} to model the intertwining operators
$\mathcal{A}_v$ by elements of $\mathcal{H}_q$. We note that by results of
Iwahori and Matsumoto~{\cite{IwahoriMatsumoto}} the convolution algebra of
$J$-biinvariant functions on $K$ is isomorphic to $\mathcal{H}_q$; in this
isomorphism $T_i$ corresponds to the characteristic function of the double
coset $J s_i J$.

The space $V_{\mathbf{z}}^J$ of Iwahori fixed vectors, like $\mathcal{H}_q$,
has dimension $| W |$. We define a vector space isomorphism
$\alpha_{\mathbf{z}} : V_{\mathbf{z}}^J \longrightarrow \mathcal{H}_q$ as
follows. Let $f \in V_{\mathbf{z}}^J$. Define $\alpha_{\mathbf{z}} (f)$ to
be the function on $G (F)$ defined by
\[ \alpha_{\mathbf{z}} (f) (g) = \left\{\begin{array}{ll}
     f (g^{- 1}) & \text{if $g \in K$,}\\
     0 & \text{otherwise.}
   \end{array}\right. \]
The function $\alpha_{\mathbf{z}} (f)$ is clearly constant on double cosets
$J g (K \cap B (F))$ and is supported in $K$; but it is easy to see using the
Iwahori factorization of $J$ that if $g \in K$ then $J g (K \cap B (F)) = J g
J$. Thus $\alpha_{\mathbf{z}} (f)$ is in $\mathcal{H}_q$ interpreted as the
convolution algebra. It is easy to see that $\alpha_{\mathbf{z}} :
V_{\mathbf{z}}^J \longrightarrow \mathcal{H}_q$ is a linear isomorphism.

Now let $\phi \in \mathcal{H}_q$. Then we claim that
\begin{equation}
  \label{convohom} \alpha_{\mathbf{z}} (\pi (\phi) f) = \phi \ast
  \alpha_{\mathbf{z}} (f) .
\end{equation}
Indeed, applied to $g \in G$, both sides vanish unless $g \in K$. Assuming
this,
\[ \alpha_{\mathbf{z}} (\pi (\phi) f) (g) = (\pi (\phi) f) (g^{- 1}) =
   \int_K \phi (k) f (g^{- 1} k) \, d k = \int_K \phi (k)
   (\alpha_{\mathbf{z}} f) (k^{- 1} g) \, d k \]
from which (\ref{convohom}) follows. We may express (\ref{convohom}) by saying
that $\alpha_{\mathbf{z}}$ intertwines the action of $\mathcal{H}_q$ on
$V_{\mathbf{z}}^J$ with the left regular representation of $\mathcal{H}_q$
on itself.

Now we use the maps $\alpha_{\mathbf{z}}$ to transfer the intertwining
operator $\mathcal{A}_w : V_{\mathbf{z}} \longrightarrow V_{w\mathbf{z}}$
to a map $A_w : \mathcal{H}_q \longrightarrow \mathcal{H}_q$ by requiring the
following diagram to commute:
\[\begin{tikzcd}
V_{\mathbf{z}} \arrow{r}{\mathcal{A}_v}\arrow{d}{\alpha_{\mathbf{z}}} & V_{v\mathbf{z}}\arrow{d}{\alpha_{\mathbf{z}}}\\
\mathcal{H}_q\arrow{r}{A_v} & \mathcal{H}_q 
\end{tikzcd}\]

Using (\ref{convohom}) the map $A_w$ is an intertwining operator for the left
regular representation of $\mathcal{H}_q$, and therefore $A_w (\phi) = A_w
(\phi \ast 1_{\mathcal{H}_q}) = \phi \ast \mu_{\mathbf{z}} (w)$ where
$\mu_{\mathbf{z}} (w) = A_w (1_{\mathcal{H}_q}) \in \mathcal{H}_q$. This
much is in {\cite{BumpNakasuji}}. In {\cite{BumpNakasujiKL}} it was shown that
\begin{equation}
  \label{musume} \mu_{\mathbf{z}} (w) = \sum_{y \leqslant w} q^{- \ell(y)}
  \overline{r_{y, w}} \, T_{y^{- 1}}
\end{equation}
where $r_{u, w}$ is as in Theorem~\ref{sigmabnform}.
We may now compute the pairing $\langle \mathcal{A}_v \psi^{\mathbf{z}}_u,
\hat{\psi}^{v\mathbf{z}}_w \rangle$. This equals
\[ \sum_{x \leqslant u} \langle \mathcal{A}_v \phi_x, \hat{\psi}_w \rangle =
   \sum_{x \leqslant u} \Lambda_w (T_x \mu_{\mathbf{z}} (v)) . \]
Now substituting (\ref{musume}) and using (\ref{innerplength}) we
obtain~(\ref{sigmabnform}).
\end{proof}

\section{Properties of $\sigma(u,v,w)$}

In the last section we introduced the linear forms $\Lambda_w$
on the Hecke algebra, and the related values $\Theta(x,y,w)$.
These were defined in (\ref{theta_def}) and appeared in the formula (\ref{sigmabnform}) for $\sigma$.
Therefore our first task is to prove some properties of $\Theta$.
It is obviously a polynomial in~$q$.

\begin{proposition}
  \label{thetalower}The polynomial $\Theta (x, y, w)$ is divisible by
  $q^{(\ell (x) + \ell (y) + \ell (x y^{- 1})) / 2}$. Its degree is bounded
  by $\ell(x)+\ell(y)$.
\end{proposition}

Note that $\ell (x y^{- 1}) \equiv \ell (x) + \ell (y)$ modulo $2$, so the
exponent here is an integer.

\begin{proof}
  Let us evaluate $\Theta (x, y, w) = \Lambda_w (T_x T_{y^{- 1}})$ from the
  definition. Let $r = \ell (x)$, $s = \ell (y)$ and $t = \ell (x y^{- 1})$.
  Taking reduced words $x = s_{i_1} \cdots s_{i_r}$ and $y = s_{j_1} \cdots
  s_{j_s}$ we have $T_x T_{y^{- 1}} = T_{i_1} \cdots T_{i_r} T_{j_s} \cdots
  T_{j_1}$. Using the braid relations and the quadratic relations we may
  express this as a sum of terms of the form $q^a (q - 1)^b T_{k_1} \cdots
  T_{k_u}=q^a (q-1)^b T_z$, where $z=s_{k_1}\cdots s_{k_u}$ is a reduced
  expression, as follows. Each time we use a quadratic relation $T_i^2 = (q - 1)
  T_i + q$ we replace $T_i^2$ by either $(q - 1) T_i$ or $q$; the number of
  times we take $(q - 1) T_i$ is $b$, and the number of times we take $q$ is
  $a$. The total number of such quadratic relation applications is
  $\frac{1}{2} (r + s - t)$, so $a + b = \frac{1}{2} (r + s - t)$. Moreover $u
  = b + t$ because $t$ of the $T_i$ in the expression $q^a (q - 1)^b T_{k_1}
  \cdots T_{k_u}$ come from the original expression and an additional $b$ come
  from the $b$ terms $(q - 1) T_i$ in the quadratic relations. Now applying
  $\Lambda_w$ to $q^a (q - 1)^b T_{k_1} \cdots T_{k_u}$ we obtain a polynomial
  in $q$, whose term of lowest degree is $(- 1)^b q^{a + u}$. The degree of
  this is $a + u = a + b + t = \frac{1}{2} (r + s + t) .$ So this is the
  smallest possible exponent for a monomial in $\Theta (x, y, w)$. The
  degree bound may be obtained similarly.
\end{proof}

If $\xi = \sum_{w \in W} a_w T_w \in \mathcal{H}_q$ we define the
{\textit{support $\supp (\xi)$}} of $\xi$ to be $\{ w \in W|a_w \neq 0
\}$.

\begin{proposition}
  \label{supptt}The set $\supp (T_u T_v)$ has minimal element $u v$
  and maximal element $u\circ v$.
\end{proposition}

\begin{proof}
  Let $s$ be a simple reflection that is a left descent of $v$, so $v = s
  v'$ with $\ell (v') < \ell (v)$. Thus $T_v = T_s T_{v'}$ and $T_u T_v = T_u
  T_s T_{v'}$.
  
  If $u s > u$ then $T_u T_v = T_u T_s T_{v'} = T_{u s} T_{v'}$. By induction
  on $\ell (v)$ the minimal element of $\supp (T_{u s} T_{v'})$ is $u s
  v' = u v$, and we are done in this case.
  
  If $u s < u$ then $T_u T_s = (q - 1) T_u + q T_{u s}$ and, again using
  induction on $\ell (v)$:
  \[ \min (\supp (T_u T_s T_{v'})) = \min (\supp (T_u T_{v'}) \cup
     \supp (T_{u s} T_{v'})) = \min (u v', u s v') = u s v' = u v. \]
  
  This shows that the minimal element of $\supp(T_uT_v)$ is $u v$. The
  fact that the maximal element is $u\circ v$ may be proved similarly.
\end{proof}

Proposition~\ref{thetalower} gives a lower bound for the exponents of $q$ in the
polynomial $\Theta (x, y, w)$. It is not sharp since the actual exponent
depends on $w$. For example if $x = y = s$ is a simple reflection, then
$\Theta (s, s, 1) = q$ and the bound in Proposition~\ref{thetalower} is sharp.
However $\Theta (s, s, s) = q^2$ and the bound is not sharp in this case.

It is an empirical observation that
the polynomial $\Theta (x, y, w)$ is often a power of~$q$. For example:

\begin{proposition}
  \label{easycasestheta}
(i) If $xy^{-1}\leqslant w$, then $\Theta(x,y,w)$ is a nonzero polynomial
whose value at $q=1$ equals~1. On the other hand, if $xy^{-1}$ is not $\leqslant w$, 
then $\Theta(x,y,w)=0$.

\smallskip\noindent
(ii) If $x\circ y^{-1}\leqslant w$, then $\Theta(x,y,w)=q^{\ell(x)+\ell(y)}$.
\end{proposition}

\begin{proof}
To prove (i), the vanishing of $\Theta(x,y,w)$ unless $xy^{-1}\leqslant w$ follows
from Proposition~\ref{supptt}. Assume that $xy^{-1}\leqslant w$. On specializing
$q\to 1$, the relations defining $\mathcal{H}_q$ become the Coxeter relations in $W$,
so $\Lambda_w(T_xT_{y^{-1}})$ has the same limit as $\Lambda_w(T_{xy^{-1}})=q^{\ell(xy^{-1})}$,
which is $1$. In particular $\Theta(x,y,w)$ is a nonzero polynomial.

We prove (ii).  First consider the case where $w=w_0$. Then $\Lambda_{w_0}$ is a homomorphism,
so
\[\Theta(x,y,w)=\Lambda_{w_0}(T_xT_{y^{-1}})=\Lambda_{w_0}(T_x)\Lambda_{w_0}(T_{y^{-1}})=
q^{\ell(x)}q^{\ell(y)}.\]
Now in the general case, then by Proposition~\ref{supptt},
every element of the support of $T_xT_{y^{-1}}$ is
$\leqslant x\circ y^{-1}$, so $\Lambda_w(T_xT_{y^{-1}})=\Lambda_{w_0}(T_xT_{y^{-1}})$,
and the statement follows.
\end{proof}

See Proposition~\ref{thetaeval} for another case where $\Theta(x,y,w)$
is known to be a power of~$q$. It is not hard to show that assuming
$xy^{-1}\leqslant w$, then when the polynomial
$\Theta(x,y,w)$ is evaluated at $q=1$ the result is~$1$. In particular
$\Theta(x,y,w)$ is nonzero in this case.

\begin{conjecture}
  Assume that the Kazhdan-Lusztig polynomial $P_{x y^{- 1},w} = 1$. Assume 
  also that $xy^{-1}\leqslant w$. (Otherwise $\Theta(x,y,w)=0$ by Proposition~\ref{easycasestheta}.) 
  Then $\Theta (x, y, w)$ is a power of~$q$.
\end{conjecture}

\noindent
This has been verified by computer calculations for Cartan types $A_3$ and~$B_3$.

\begin{lemma}
  \label{thetanv}If $x \geqslant u$, $y \leqslant v$ and $\Theta (x, y, w)
  \neq 0$ then
  \begin{equation}
  \label{thetanvcondition}
  U_{w^{- 1}} \downarrow u \leqslant U_{w^{- 1}} \downarrow x \leqslant y
     \leqslant v.
  \end{equation}
\end{lemma}

\begin{proof}
  By Proposition~\ref{supptt}, $\Theta (x, y, w) = 0$ unless $w \leqslant x
  y^{- 1}$. This means that $y x^{- 1} \leqslant w^{- 1}$, that is, $y \in [1,
  w^{- 1}] x$. By Theorem~\ref{translatemin}, the set $[1, w^{- 1}] x$ has a
  unique minimal element $U_{w^{- 1}} \downarrow x$. Thus we obtain $U_{w^{-
  1}} \downarrow x \leqslant y \leqslant v$. The remaining inequality $U_{w^{-
  1}} \downarrow u \leqslant U_{w^{- 1}} \downarrow x$ follows from
  Proposition~\ref{downmonotone}.
\end{proof}
Our next result shows that with $u,w$ fixed, $v=U_{w^{-1}}\downarrow u$ 
is the smallest $v$ such that $\sigma(u,v,w)\neq 0$, and moreover
this value is a polynomial in $q$. We define
\[\sigma_0(u,w)=\sigma(u,U_{w^{-1}}\downarrow u,w).\]
This partially proves Theorem~\ref{maintheorem}.

\begin{proposition}
  \label{firstpart}
  We have $\sigma (u, v, w) = 0$ unless $v \geqslant U_{w^{- 1}} \downarrow
  u$. In the minimal case where $v = U_{w^{- 1}} \downarrow u$, the
  value $\sigma_0 (u, w)$ is a
  polynomial in $q$ (independent of $\mathbf{z}$).
\end{proposition}

\begin{proof}
  We make use of (\ref{sigmabnform}). With $x\geqslant u$, $y\leqslant v$,
  by Lemma~\ref{thetanv}, $\Theta (x, y, w) = 0$ unless (\ref{thetanvcondition})
  is satisfied. Therefore $\sigma(u,v,w)=0$ unless $U_{w^{-1}}\downarrow u \leqslant v$,
  proving the first assertion. 

  Moreover if $v=U_{w^{- 1}}\downarrow u$, then
  (\ref{thetanvcondition}) implies that $y=v$. Therefore $r_{y,v}=1$.
  Also by Proposition~\ref{supptt}, $\Theta(x,v,w)$ vanishes unless
  $xv^{-1}\leqslant w$.
  Hence (\ref{sigmabnform}) simplifies to 
  \begin{equation}
  \label{sigmao}
  \sigma_0 (u, w) = q^{- \ell(v)} \sum_{\substack{x \geqslant u\\xv^{-1}\leqslant w}} \Theta (x, v, w) ,\qquad v=U_{w^{-1}}\downarrow u.
  \end{equation}

  Since $\Theta(x,v,w)$ does not involve $\z$, it is a polynomial in~$q$. 
  But since we are dividing by $q^{\ell(v)}$ in (\ref{sigmao}), we
  must show that $\Theta(x,v,w)$ is divisible by $q^{\ell(v)}$.
  Actually we will show that it is divisible by
  $q^{\ell(u)}$, a stronger result since $v\leqslant_L u$.
  With $x\geqslant u$ we have $xv^{-1}\in[u,w_0]v^{-1}$ so
  by Theorem~\ref{translatemin} we have
  \[xv^{-1}\geqslant u\downarrow U_{v^{-1}}=uv^{-1}\]
  where the last equality follows since $v\leqslant_L u$.
  Now we apply Proposition~\ref{thetalower} and use
  \[\frac12\bigl(\ell(x)+\ell(v)+\ell(xv^{-1})\bigr)\geqslant
  \frac12\bigl(\ell(u)+\ell(v)+\ell(uv^{-1})\bigr)=\ell(u).\]
\end{proof}

If $\alpha$ is a positive root, let $r_{\alpha} \in W$ be the associated
reflection. If $u \leqslant v$ in the Bruhat order, define:
\begin{equation}
 \label{suvdwf} S (u, v) = \{ \alpha \in \Phi^+ |u \leqslant vr_{\alpha} <
   v \} .
\end{equation}
We remind the reader of the Deodhar inequality (\ref{deodharineq}),
valid if the inverse Kazhdan-Lusztig polynomial $Q_{u, v} : = P_{w_0
v, w_0 u}$ equals $1$. We will generalize this definition slightly and denote
\begin{equation}
 \label{suvwdef} S (u, v, w) : = S (U_{w^{- 1}} \downarrow u, v) = \{ \alpha
   \in \Phi^+ |U_{w^{- 1}} \downarrow u \leqslant vr_{\alpha} < v \} . 
\end{equation}

Bump and Nakasuji conjectured, and Aluffi, Mihalcea, Sch{\"u}rmann and Su
proved the following result.

\begin{theorem}
  \label{poledet}Let $u \leqslant v$. Then
  \[ r_{u, v} \prod_{\alpha \in S (u, v)} (1 -\mathbf{z}^{\alpha}) \]
  is analytic for all $\mathbf{z} \in \hat{T} (\mathbb{C})$.
\end{theorem}

\begin{proof}
  See {\cite{AMSS}}, Theorem~10.4.
\end{proof}

We note the difficulty of Theorem~\ref{poledet}. An obvious approach is to
prove this recursively from Theorem~\ref{rrecursion}. This is the basis of
partial results in~{\cite{BumpNakasujiKL}}. However this does not produce
exactly the right set of roots $\alpha$ such that $1 -\mathbf{z}^{\alpha}$
appears in the denominator, because some roots that appear in the terms of the
recursive formula actually cancel; these cancellations are not easy to
prove.

\begin{theorem}
  The function
  \[ \sigma (u, v, w) \prod_{\alpha \in S (u, v, w)} (1
     -\mathbf{z}^{\alpha}) \]
  is analytic for all $\mathbf{z} \in \hat{T} (\mathbb{C})$.
\end{theorem}

\begin{proof}
  We make use of (\ref{sigmabnform}) and Theorem~\ref{poledet}. We see that
  the possible roots $\alpha$ such that $1 -\mathbf{z}^{\alpha}$ appears in
  the denominator lie in
  \[ \bigcup_{\substack{ x \geqslant u\\
       y \leqslant v\\
       \Theta (x, y, w) \neq 0}}
    S (y, v) . \]
  Now by Lemma~\ref{thetanv}, we must have $U_{w^{- 1}} \downarrow u \leqslant
  y$ in this union, so $S (y, v) \leqslant S (u, v, w)$ from the
  definition~(\ref{suvwdef}).
\end{proof}

\section{Proof of Theorem~\ref{maintheorem}}

Let $u, w \in W$. In this section we will take $v = v_{\min} (u, w) = U_{w^{-
1}} \downarrow u$ and prove Theorem~\ref{maintheorem}.

We will change our notation slightly for simple reflections. In previous
sections we denoted by $\{ s_1, \cdots, s_r \}$ the set of simple reflections,
so a reduced word for a Weyl group element $v$ would be written $s_{i_1}
\cdots s_{i_k}$ where $\{ i_1, \cdots, i_k \}$ is some sequence of indices. In
the following arguments this notation would be cumbersome, so we dispense with
the double subscripts and write a reduced expression for $v$ as $s_1 \cdots
s_k$. Thus $k = \ell (v)$.

Let $m = u v^{- 1}$. Then $m$ is the ``mixed meet'' of
Theorem~\ref{thm:mixedmeet}, that is, the maximal element of $w$ such that $m
\leqslant_R u$ and $m < w$. We will write $u = m v = m s_1 \cdots s_k$.
Because $v = U_{w^{- 1}} \downarrow u$ we have $v \leqslant_L u$, that is
$\ell (u) = \ell (m) + \ell (v)$. Therefore
\begin{equation}
  \label{uchain} m < m s_1 < m s_1 s_2 < \cdots < m s_1 \cdots s_k = u.
\end{equation}
We will denote
\[ \llbracket u, w \rrbracket = \{ z \in W|u \leqslant z, z v^{- 1} \leqslant w \} . \]

\begin{figure}[h]
\[\begin{tikzpicture}
\draw (0,0) -- (-3,3);
\draw (1,1) -- (-2,4);
\draw (2,2) -- (-1,5);
\draw (0,0) -- (2,2);
\draw (-1,1) -- (1,3);
\draw (-2,2) -- (0,4);
\draw (-3,3) -- (-1,5);
\path[fill=white] (0,0) circle (.3);
\path[fill=white] (-1,1) circle (.3);
\path[fill=white] (-2,2) circle (.3);
\path[fill=white] (-3,3) circle (.3);
\path[fill=white] (1,1) circle (.3);
\path[fill=white] (0,2) circle (.3);
\path[fill=white] (-1,3) circle (.3);
\path[fill=white] (-2,4) circle (.3);
\path[fill=white] (2,2) circle (.3);
\path[fill=white] (1,3) circle (.3);
\path[fill=white] (0,4) circle (.3);
\path[fill=white] (-1,5) circle (.3);
\node at (0,0) {$m$};
\node at (-1,1) {$ms_1$};
\node at (-2,2) {$ms_1s_2$};
\node at (-3.25,3){$u=ms_1s_2s_3$};
\node at (1,1) {$zs_3s_2s_1$};
\node at (0,2) {$zs_3s_2$};
\node at (-1,3) {$zs_3$};
\node at (-2,4){$z$};
\node at (2,2) {$w$};
\node at (1,3) {$ws_1$};
\node at (0,4) {$ws_1s_2$};
\node at (-1,5){$ws_1s_2s_3$};
\end{tikzpicture}\]
\caption{The Weyl group elements in Proposition~\ref{ladder} form a ladder.
Here is the case where $k=3$.}
\label{figladder}
\end{figure}
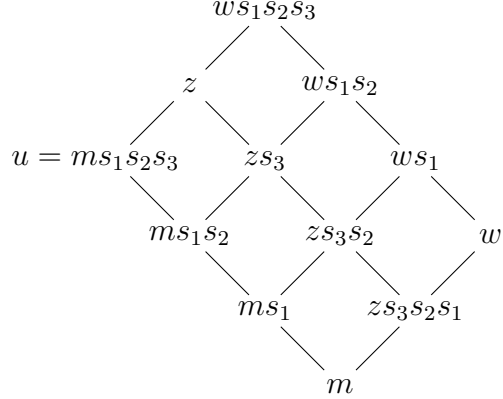

\begin{proposition}
  \label{ladder}Let $u, w \in W$ and let $v = v_{\min} (u, w)$. Then $\ell (w
  v) = \ell (w) + \ell (v)$. Moreover $\llbracket u, w \rrbracket = [u, w v]$.
  Let $k = \ell (v)$ and let $v = s_1 \cdots s_k$ be a reduced expression.
  Then
  \begin{equation}
    \label{keytow} w < w s_1 < w s_1 s_2 < \cdots < w s_1 \cdots s_k = w v.
  \end{equation}
  Suppose that $z \in [u, w v]$. Then
  \begin{equation}
    \label{zchain} z > z s_k > z s_k s_{k - 1} > \cdots > z s_k \cdots s_1 = z
    v^{- 1} .
  \end{equation}
  We have
  \begin{equation}
    \label{fullladder} u s_k \cdots s_{r + 1} = m s_1 \cdots s_r \leqslant z
    s_k \cdots s_{r + 1} \leqslant w s_1 \cdots s_r 
  \end{equation}
  for $0 \leqslant r \leqslant k$.
\end{proposition}

\begin{proof}
  The inequalities asserted by the Proposition may be envisioned
  as forming a ladder, as in Figure~\ref{figladder}. We will
  ascend the ladder to prove the inequalities satisfied by
  the $ws_1\cdots s_r$, then descend the ladder to prove
  the inequalities satisfied by the $zs_ks_{k-1}\cdots s_{r+1}$.
  
  Thus we will by induction for $0 \leqslant r \leqslant k$ that
  \begin{equation}
    \label{keystp} m s_1 \cdots s_r \leqslant w s_1 \cdots s_r .
  \end{equation}
  If $r = 0$ this is true since $m = u v^{- 1} < w$. Arguing inductively,
  assume that this is true for $r < k$; we will prove that it is true for $r +
  1$. The first step is to show
  \begin{equation}
    \label{uwfs} w s_1 \cdots s_r < w s_1 \cdots s_r s_{r + 1} .
  \end{equation}
  If not, $w s_1 \cdots s_r < w s_1 \cdots s_r s_{r + 1} .$ Now since $m s_1
  \cdots s_r \leqslant w s_1 \cdots s_r$ and $m s_1 \cdots s_r < m s_1 \cdots
  s_r s_{r + 1}$, the lifting property of the Bruhat order implies that
  \[ m s_1 \cdots s_r s_{r + 1} \leqslant w s_1 \cdots s_r . \]
  Using the monotonicity property Proposition~\ref{downmonotonebis} we then
  have
  \[ U_{(w s_1 \cdots s_r)^{- 1}} \downarrow u \leqslant U_{(m s_1 \cdots s_{r
     + 1})^{- 1}} \downarrow u. \]
  However
  \[ U_{(w s_1 \cdots s_r)^{- 1}} \downarrow u = U_{(s_1 \cdots s_r)^{- 1}}
     \downarrow U_{w^{- 1}} \downarrow u = U_{(s_1 \cdots s_r)^{- 1}}
     \downarrow v = s_{r + 1} \cdots s_k\,, \]
  while
  \[ U_{(m s_1 \cdots s_{r + 1})^{- 1}} \downarrow u = U_{(m s_1 \cdots s_{r +
     1})^{- 1}} \downarrow m s_1 \cdots s_k = s_{r + 2} \cdots s_k\,, \]
  which is interpreted as $1$ if $k = r - 1$. Thus we have proved that $s_{r
  + 1} \cdots s_k \leqslant s_r \cdots s_k$, which is a contradiction.
  Therefore we have proved~(\ref{uwfs}).
  
  Now using (\ref{keystp}) and (\ref{uwfs}) and the inequality $m s_1 \cdots
  s_r < m s_1 \cdots s_{r + 1}$ from (\ref{uchain}), the lifting property of
  the Bruhat order implies that $m s_1 \cdots s_{r + 1} \leqslant w s_1 \cdots
  s_{r + 1}$, which is (\ref{keystp}) for $r + 1$. This completes the
  induction, so we have now proved (\ref{keystp}) for all $r$.
  
  Note that we have also proved (\ref{keytow}) from the
  inequalities~(\ref{uwfs}).
  
  Now let $z \in [u, w v]$. We now refine (\ref{keystp}) to the inequality
  (\ref{fullladder}). This time we argue by {\textit{downwards}} induction, the
  initial case being $r = k$, where (\ref{fullladder}) becomes our assumption
  $u \leqslant z \leqslant w v$. Assuming (\ref{fullladder}) for $r > 0$, we
  prove it for $r - 1$. First we need to show that
  \begin{equation}
    \label{zdescent} z s_k \cdots s_{r + 1} > z s_k \cdots s_r
  \end{equation}
  If not, $z s_k \cdots s_{r + 1} < z s_k \cdots s_r$. We have also $z s_k
  \cdots s_{r + 1} \leqslant w s_1 \cdots s_r$ by (\ref{fullladder}) and $w
  s_1 \cdots s_r > w s_1 \cdots s_{r - 1}$, so by the lifting property of the
  Bruhat order we obtain $z s_k \cdots s_{r + 1} \leqslant w s_1 \cdots s_{r -
  1}$, and using the first inequality in (\ref{fullladder}) we get $m s_1
  \cdots s_r \leqslant w s_1 \cdots s_{r - 1}$. Now using
  Proposition~\ref{downmonotonebis} we have
  \[ U_{(w s_1 \cdots s_{r - 1})^{- 1}} \downarrow u \leqslant U_{(m s_1
     \cdots s_r)} \downarrow u, \]
  that is $s_r \cdots s_k \leqslant s_{r + 1} \cdots s_k$, which is a
  contradiction, proving (\ref{zdescent}). Now using (\ref{zdescent}),
  (\ref{fullladder}) and (\ref{keytow}), the lifting property of the Bruhat
  order implies that $u s_k \cdots s_{r + 1} = m s_1 \cdots s_r \leqslant z
  s_k \cdots s_{r + 1} \leqslant w s_1 \cdots s_r$ which is a contradiction.
  Therefore we have proved (\ref{zdescent}).
  
  Given (\ref{zdescent}), (\ref{fullladder}), (\ref{uchain}) and
  (\ref{keytow}), both inequalities in
  \[ u s_k \cdots s_r = m s_1 \cdots s_{r - 1} \leqslant z s_k \cdots s_r
     \leqslant w s_1 \cdots s_{r - 1} \]
  follow from the lifting property of the Bruhat order. This is
  (\ref{fullladder}) for $r - 1$, completing the proof of (\ref{fullladder})
  by induction. Note that we have also proved (\ref{zchain}).
  
  Taking $k = 1$ in (\ref{fullladder}) gives $z v^{- 1} \leqslant w$.
  Therefore $\llbracket u, w \rrbracket \subseteq [u, w v]$. To prove the
  opposite inclusion, assume that $z \in \llbracket u, w \rrbracket$. We must
  show that $z \leqslant w v$. Indeed $z = (z v^{- 1}) v \leqslant w \circ v$.
  But $w \circ v = w v$ follows from (\ref{keytow}), and so we are done.
\end{proof}

\begin{lemma}
  \label{ionlyv}Suppose that $v = v_{\min} (u, w)$ and that $z \in [u, w v]$.
  Suppose furthermore that $v' \leqslant v$ and that $z (v')^{- 1} \leqslant
  w$. Then $v' = v$.
\end{lemma}

\begin{proof}
  Since $v' z^{- 1} \leqslant w^{- 1}$ we have $v' \in [1, w^{- 1}] z$ and so
  by Theorem~\ref{translatemin} we have $U_{w^{- 1}} \downarrow z \leqslant
  v'$. Since $u \leqslant z$, Proposition~\ref{downmonotone} implies that $v =
  U_{w^{- 1}} \downarrow u \leqslant v'$, so $v' = v$.
\end{proof}

\begin{proposition}
  \label{thetaeval}Suppose that $v = v_{\min} (u, w)$ and that $z \in [u, w
  v]$. Then $\Theta (z, v, w) = q^{\ell (z)}$.
\end{proposition}

\begin{proof}
  Recall that $\Theta (z, v, w) = \Lambda_w (T_z T_{v^{- 1}})$. Expanding $T_z
  T_{v^{- 1}} = T_z T_{s_k} \cdots T_{s_1}$ using the fact that
  \begin{equation}
  \label{tscases}
  T_y T_s = \left\{\begin{array}{ll}
       T_{y s} & \text{if $y s > y$},\\
       (q - 1) T_y + q T_{y s} & \text{if $y s < y$},
     \end{array}\right.
  \end{equation}   
  it is clear that 
  \[\supp (T_z T_{v^{- 1}}) \subseteq \{T_{z (v')^{- 1}}\,|\,v'\leqslant v\}\;.\]
  By Lemma~\ref{ionlyv}, only $v' = v$ can contribute
  to the value of $\Lambda_w$. Moreover the only way to obtain $T_{z v^{- 1}}$
  is to select $q T_{y s}$ in the descent case of (\ref{tscases}) each time. Thus
  \[ T_z T_{v^{- 1}} = q^{\ell (v)} T_{z v^{- 1}} + \text{other terms} \]
  where the other terms are linear combinations of $T_{z (v')^{- 1}}$ with $v'
  < v$. Note that $\ell(v)+\ell(z v^{-1})=\ell(z)$ by (\ref{zchain}) in
  Proposition~\ref{ladder}. Thus applying $\Lambda_w$ we get
  \[ \Theta (z, v, w) = \Lambda_w (q^{\ell (v)} T_{z v^{- 1}}) = q^{\ell (v) +
     \ell (z v^{- 1})} = q^{\ell (z)} . \qedhere\]
\end{proof}

\begin{proof}[Proof of Theorem~\ref{maintheorem}]
  A portion of Theorem~\ref{maintheorem} is contained in
  Proposition~\ref{firstpart}. To finish the proof we need
  to extablish (\ref{sigmabruhat}).
  By Proposition~\ref{ladder}, we may rewrite (\ref{sigmao}) as
  \[ q^{- \ell (v)} \sum_{z \in [u, w v]} \Theta (z, v, w), \]
  and now (\ref{sigmabruhat}) follows from Proposition~\ref{thetaeval}.
\end{proof}
\bibliographystyle{habbrv}
\bibliography{matinter}
\end{document}